\documentclass[11pt]{amsart}
\usepackage{texmaX,semmaX,semtkzX}
\usepackage{makecell}
\usepackage{multirow}

\usepackage{booktabs} 
\usepackage{mathtools}
\usepackage{listings}

\usepackage{verbatim}  
\usepackage{enumitem}

\usepackage{hyperref}

\usepackage{lmodern}
\usepackage{minted} 

\usepackage[margin=1in]{geometry}

\usepackage{amsthm}
\newtheorem*{unnumberedquestion}{Question}

\DeclareFontFamily{U}{wncy}{}
\DeclareFontShape{U}{wncy}{m}{n}{<->wncyr10}{}
\DeclareSymbolFont{mcy}{U}{wncy}{m}{n}
\DeclareMathSymbol{\Sha}{\mathord}{mcy}{"58}

\usepackage{tikz}
\usetikzlibrary{arrows.meta}
\usetikzlibrary{arrows.meta}
\usetikzlibrary{positioning, arrows.meta}

\usepackage[utf8]{inputenc}

\usepackage[table]{xcolor}

\usepackage{placeins}

\usepackage{float}
\usepackage{caption}

\usepackage{xcolor}
\usepackage{ulem}

\begin{document}

\title[BSD twins of elliptic curves]{Infinitely many pairs of non-isomorphic elliptic curves sharing the same BSD invariants}

\author{Asuka Shiga}
\address{Mathematical Institute, Graduate School of Science, Tohoku University, 6-3 Aramaki Aza Aoba, Aoba-ku, Sendai, Miyagi 980-8578, Japan.}
\email{otheiio323.com@gmail.com}

\begin{abstract}{Let \(E/\mathbb{Q}\) be an elliptic curve. The Birch and
Swinnerton--Dyer (BSD) conjecture relates the leading coefficient of the
Taylor expansion of the \(L\)-function of \(E/\mathbb{Q}\) at \(s=1\) to
arithmetic invariants of \(E\), such as its Mordell--Weil group, its
Tate--Shafarevich group, its Tamagawa numbers, its regulator, and its
real period. We call two non-isomorphic elliptic curves over \(\mathbb{Q}\) BSD twins if
they have the same \(L\)-function and the same arithmetic data underlying the
BSD invariants appearing in the BSD conjecture, with the
Mordell--Weil group and the Tate--Shafarevich group compared as groups. We exhibit a family of BSD twins for which the corresponding
pairs of \(j\)-invariants are pairwise distinct. We further prove that, even after imposing equality of the Kodaira symbols at every prime and equality of the minimal discriminants, infinitely many BSD twins still exist.
}

\end{abstract}



\maketitle

\tableofcontents

\section{Introduction}\label{sec1}

\subsection{The Question}

Let $E/\mathbb{Q}$ be an elliptic curve defined over $\mathbb{Q}$. We define the following 6-tuple as Birch and Swinnerton-Dyer (BSD) data:

\[
\mathrm{BSD}(E/\mathbb{Q}) \stackrel{\mathrm{def}}{=}
\bigl(L(E,s),\,E(\mathbb{Q}),\,\mathrm{Reg}(E/\mathbb{Q}),\,\Omega_E,\, (c_{E/\mathbb{Q}_p})_p,\,\Sha(E/\mathbb{Q})\bigr).
\]

Here, the definition of each component is as follows.
\begin{itemize}

\item $L(E,s)$: The $L$-function of $E/\mathbb{Q}$.

\item  $E(\mathbb{Q})$: The Mordell--Weil group of $E/\mathbb{Q}$, which is isomorphic to $\mathbb{Z}^{\text{rank}(E/\mathbb{Q})}\times E(\mathbb{Q})_{\text{tor}}$. 

\item $\text{Reg}(E/\mathbb{Q})$: The regulator of $E/\mathbb{Q}$.

\item $\Omega_{E}$: The real period of $E$, given by $\Omega_E = \int_{E(\mathbb{R})} |\omega|$ where $\omega$ is the invariant differential of a global minimal model of $E/\Bbb{Q}$. 

\item $c_{E/\Bbb{Q}_p}$: The Tamagawa number at a prime number $p$.

\item
$\Sha(E/\Bbb{Q})$: The Tate-Shafarevich group of $E/\Bbb{Q}$, conjecturally finite.

\end{itemize}

Let $\mathrm{rank}_{\mathrm{an}}(E/\Bbb{Q})=\operatorname{ord}_{s=1}L(E,s)$ be analytic rank of $E/\Bbb{Q}$. The weak form of the BSD
conjecture predicts that $\mathrm{rank}_{\mathrm{an}}(E/\Bbb{Q})=\operatorname{rank}(E/\mathbb{Q})$. The full BSD conjecture further predicts the finiteness of
\(\Sha(E/\mathbb{Q})\) and the leading coefficient formula
\[
\lim_{s\to 1}\frac{L(E,s)}{(s-1)^r}
=
\frac{
\Omega_E\,\operatorname{Reg}(E/\mathbb{Q})\,
\#\Sha(E/\mathbb{Q})\prod_p c_{E/\mathbb{Q}_p}
}{
\#E(\mathbb{Q})_{\mathrm{tor}}^2
},
\qquad r=\operatorname{rank}E(\mathbb{Q}).
\]

When $E_1 \cong E_2$ over $\mathbb{Q}$, we have $\text{BSD}(E_1/\mathbb{Q}) = \text{BSD}(E_2/\mathbb{Q})$. Here, “=” between groups means that the groups are isomorphic as groups.

We investigate the following question.

\begin{unnumberedquestion}
When $\text{BSD}(E_1/\mathbb{Q}) = \text{BSD}(E_2/\mathbb{Q})$, does $E_1 \cong E_2$ over $\mathbb{Q}$ hold?
\end{unnumberedquestion}

Except where explicitly stated otherwise, we do not assume the weak BSD
conjecture, the full BSD conjecture, or the finiteness of the Tate--Shafarevich groups.

Having the same $L$-function is equivalent to being isogenous over $\mathbb{Q}$ by Faltings's theorem; see \cite[Chapter~V, Section~4, Corollary~4.2]{milne}. Thus, we should choose $(E_1,E_2)$ to be from the same isogeny class over $\mathbb{Q}$.

It is known that the quantity \[\frac{\Omega_E \cdot \text{Reg}(E/\mathbb{Q}) \cdot \# \Sha(E/\mathbb{Q}) \cdot \prod_p c_{E/\Bbb{Q}_p} }{ \# E(\mathbb{Q})_{\text{tor}}^2}\] remains invariant under isogeny under the assumption that $\Sha(E/\Bbb{Q})$ is finite (\cite{Tate}, Theorem 2.1), though none of the individual terms in the product need be the same for isogenous curves.

Even when the isogeny class contains more than one elliptic curve, this question has an affirmative answer in the following example.

\begin{proposition}
There exists a unique elliptic curve $E$ defined over $\Bbb{Q}$ that has the following BSD data. 
 \FloatBarrier
  \begin{table}[htbp]
    \centering
    \label{tab:elliptic_curves_38025}
    \begin{tabular}{ll}
      \toprule
      \textbf{Elliptic curve} & $E$ \\
      \midrule
      $L$-function & \(\prod_{p\neq 2,17}\bigl(1-a_p\,p^{-s}+p^{\,1-2s}\bigr)^{-1},\quad a_p=p+1-\#E(\mathbb{F}_p), \text{where} \ E: y^2=x^3+17x.\) \\
      Mordell--Weil group & \(\mathbb{Z}/2\mathbb{Z}\) \\
      Regulator & 1 \\
      Real period & $1.8261\cdots$ \\
      Tamagawa number & \(1\,(p=2),\; 2\,(p=17)\) \\
      \(\Sha(E/\mathbb{Q})\) & \(\mathbb{Z}/2\mathbb{Z}\times \mathbb{Z}/2\mathbb{Z}\) \\
      \bottomrule
    \end{tabular}
  \end{table}
  \FloatBarrier

\end{proposition} 
\begin{proof}
A direct computation shows that \(E: y^2=x^3+17x\) has the data shown in the table. Let \(E'\) be another elliptic curve over \(\Bbb{Q}\) that has the data shown in the table (for the computation of the Tate--Shafarevich group, see \cite{Rubin}). Then \(E\) and \(E'\) are isogenous over \(\Bbb{Q}\) since they share the same \(L\)-function. Then \(E'\) is \(2\)-isogenous to \(E\) over \(\Bbb{Q}\) (by the isogeny graph of LMFDB \cite{lmfdb}), so \(E'\) is isomorphic to \(y^2=x^3-68x\). However, \(\Sha(E'/\Bbb{Q})[2]=0\). Thus, the curve is unique, namely \(E\).\end{proof}

This example shows that, even if $E$ has another isogenous pair $E'$, we can distinguish $E$ and $E'$ by the BSD data. 
Can elliptic curves in general be uniquely characterized by the BSD data?

\subsection{Main results}

We call a pair $(E_1,E_2)$ of non-isomorphic elliptic curves over $\mathbb{Q}$ a pair of BSD twins if
\[
\mathrm{BSD}(E_1/\mathbb{Q})=\mathrm{BSD}(E_2/\mathbb{Q}).
\]

We prove two complementary infinitude results for BSD twins.

For a square-free integer $D$, let $E^D$ denote the quadratic twist of an elliptic curve
$E/\mathbb{Q}$ by $D$. Although \(E^D\) is isomorphic to \(E\) over \(\overline{\mathbb{Q}}\), quadratic twisting can change the arithmetic of \(E\) over \(\mathbb{Q}\). We exploit this flexibility to construct infinitely many pairs of elliptic curves with the same BSD data.

First, we exhibit a family of geometrically distinct \(2\)-isogenous
starting pairs \((E_m,E_m')\). For each such pair satisfying the parity
hypothesis in Theorem~\ref{distinctj}, one can produce infinitely many
quadratic twists \((E_m^D,E_m'^D)\) sharing the same BSD invariants. Moreover,
for each starting pair one has \(j(E)\neq j(E')\), and the pairs of
\(j\)-invariants $\bigl(j(E_m),j(E_m')\bigr)$ are pairwise distinct as the starting pair varies.

Second, even after imposing the additional requirement that the finer
local invariants given by the Kodaira symbols and the minimal discriminant
also agree, there still exist infinitely many BSD twins over
\(\mathbb Q\) (Theorem~\ref{main}). As explained below, this condition is
extremely rare geometrically: in the \(2\)-isogenous discriminant
twin setting, the relevant starting pair is unique up to quadratic twist.

We first construct a family of pairs $(E_m,E'_m)$ whose
pairs of $j$-invariants are distinct as $m$ varies.

\begin{example}[A family of starting pairs]\label{startingpair}

Let \(m\) be either \(1\) or a prime satisfying
$\left(\frac{390}{m}\right)=-1$, and set \(n\coloneqq m^2+64\).

Let $E_m$ and $E'_m$ be the elliptic curves defined by
\[
E_m:\ y^2=x^3+390n x^2+195^2m^2n\,x, 
E'_m:\ y^2=x^3-195n x^2+16\cdot 195^2n\,x. 
\qquad
\]

We have $j(E_m)\neq j(E'_m)$. 
\end{example}

We then prove that, for infinitely many square-free integers $D$, the pairs
$(E_m^D,E_m^{\prime D})$ obtained by applying the same quadratic twist to both
curves are the desired BSD twins.

\begin{theorem}\label{distinctj}(=Theorem \ref{j})

Take $E_m,E'_m$ as in Example \ref{startingpair}. If $\mathrm{rank}_{an}(E_m/\mathbb{Q})\equiv 0 \bmod 2$ (for example, $m=7,19,23,37,43, 53$), then there exist infinitely many square-free integers $D$ such that the pair $(E^D_m,E'^D_m)$ satisfies
\[
j(E^D_m)\neq j(E'^D_m)
\quad\text{and}\quad
\operatorname{BSD}(E^D_m/\mathbb{Q})=\operatorname{BSD}(E'^D_m/\mathbb{Q}),
\]
and the pairs $(j(E^D_m),j(E'^D_m))=(j(E_m),j(E_m'))$ are pairwise distinct if we vary $m$.

\end{theorem}

However, except for the case \(m=1\), the above pairs \((E_m,E'_m)\)
can still be distinguished by finer local data not captured by the Tamagawa numbers,
namely by their Kodaira symbols or their minimal discriminants.

Even when all Tamagawa numbers coincide, the pairs need not have the same
reduction types at the bad primes. In the present family, one has
\[
\frac{\Delta_{E_m}}{\Delta_{E'_m}}=m^2.
\]
Hence the minimal discriminants coincide only in the special case \(m=1\).

In \cite{twin}, pairs of 2-isogenous elliptic curves over $\mathbb{Q}$ with the same minimal discriminants are called 2-isogenous discriminant twins, and examples of these are classified in Table 5 of \cite{twin}. This pair is unique up to quadratic twist. Indeed,

\begin{theorem}[\cite{twin}]\label{discriminanttwins}
Among pairs of elliptic curves over \(\mathbb{Q}\) connected by a
\(2\)-isogeny defined over \(\mathbb{Q}\), the pair with LMFDB labels
\(4225.h1\) and \(4225.h2\) is, up to quadratic twists, the unique pair of
discriminant twins over \(\mathbb{Q}\).
\end{theorem}

Although the starting pair is unique up to isomorphism over $\bar{\Bbb{Q}}$, 
quadratic twisting produces infinitely many BSD twins over \(\mathbb{Q}\) that
remain indistinguishable even by their minimal discriminants and Kodaira symbols.

\begin{theorem}\label{main} (=Theorem \ref{maintheorem})
There exist infinitely many pairs $(E_1,E_2)$ of non-isomorphic elliptic curves over $\mathbb{Q}$ such that $j(E_1)\neq j(E_2)$, $\text{BSD}(E_1/\mathbb{Q})=\text{BSD}(E_2/\mathbb{Q})$, and the Kodaira symbols at every prime and the minimal discriminants are the same.
\end{theorem}

In this theorem, we use the pair in Example~\ref{startingpair} corresponding to \(m=1\), the quadratic twist of $(4225.h1,4225.h2)$ by $-195$. The relationship between two main theorems (Theorem \ref{distinctj} and Theorem \ref{main}) is as follows.

\FloatBarrier

\begin{figure}[htbp]
\centering
\begin{tikzpicture}[
  >=Latex,
  every node/.style={align=center},
  box/.style={
    draw=black,
    rectangle,
    rounded corners=2pt,
    thick,
    font=\scriptsize,
    inner xsep=4pt,
    inner ysep=3pt,
    minimum height=8mm
  },
  leftbox/.style={
    box,
    text width=31mm
  },
  rightbox/.style={
    box,
    text width=35mm
  },
  machine/.style={
    draw=black,
    rectangle,
    rounded corners=3pt,
    very thick,
    font=\scriptsize,
    inner xsep=5pt,
    inner ysep=5pt,
    minimum width=25mm,
    minimum height=15mm
  },
  inclusion/.style={
    font=\fontsize{16}{16}\selectfont
  },
  mainarrow/.style={
    ->,
    very thick,
    shorten >=2pt,
    shorten <=2pt
  },
  auxarrow/.style={
    ->,
    dashed,
    thick,
    shorten >=2pt,
    shorten <=2pt
  }
]

\node[leftbox] (family) at (0,0.95)
  {Abundant family\\
   (Example~\ref{startingpair})};

\node[leftbox] (disc) at (0,-0.95)
  {Very rare \(2\)-isogenous\\
   discriminant twins\\
   (Example~\ref{example})};

\node[inclusion, rotate=90] at (0,0.1) {$\subset$};

\node[machine] (machine) at (3.75,0)
  {\textbf{Theorem~\ref{thm:twists}}\\
   twisting machine};

\node[rightbox] (bsd) at (7.55,0.95)
  {BSD twins\\
   (Theorem~\ref{distinctj})};

\node[rightbox] (strongbsd) at (7.55,-0.95)
  {BSD twins\\
   \(+\) same Kodaira symbols\\
   \(+\) same minimal discriminant\\
   (Theorem~\ref{main})};

\draw[auxarrow]
  (family.east) -- ([yshift=3mm]machine.west);

\draw[mainarrow]
  (disc.east) -- ([yshift=-3mm]machine.west);

\draw[auxarrow]
  ([yshift=5mm]machine.east) -- (bsd.west);

\draw[mainarrow]
  ([yshift=-5mm]machine.east) -- (strongbsd.west);

\end{tikzpicture}

\caption{
Theorem~\ref{thm:twists} acts as a twisting machine. Applied to the
abundant family in Example~\ref{startingpair}, it produces BSD twins over
\(\mathbb{Q}\). Applied to the very rare discriminant twins in
Example~\ref{example}, it produces infinitely many pairs over \(\mathbb{Q}\)
for which the BSD invariants, the Kodaira symbols, and the minimal discriminants
all coincide.
}
\end{figure}

\FloatBarrier

\FloatBarrier

Finally, we explain the related work relevant to the present study.

In the high-dimensional case, Jamie Bell proved the existence of pairs of abelian varieties of dimension $p-1$ (where the class number of $\mathbb{Q}(\zeta_p)$ is not $1$, so $p \geq 23$), defined over $\mathbb{Q}$, such that over every number field they share the same BSD invariants except possibly for the real period whose
equality is conditional on the BSD conjecture, the same \(n\)-Selmer
groups for all \(n\), and the same Tate modules \cite{Bell}. The pair becomes isomorphic over $\overline{\mathbb{Q}}$.
For this pair of abelian varieties, for every prime \(\ell\), there exists
an isogeny of some degree \(q\) with \((q,\ell)=1\). However, for two non-isomorphic elliptic curves over $\mathbb{Q}$, there cannot exist two or more isogenies of coprime degrees. Therefore, we need to examine the $q$-primary component of the invariants in the case of elliptic curves connected by prime degree $q$ isogenies.

An analogous phenomenon is known for number fields: Sutherland showed that there are infinitely many pairs of non-isomorphic solvably equivalent number fields (\cite{Sutherland}); in particular, all invariants appearing in the analytic class number formula coincide for such pairs.

\subsection{Outline of the proof}
First, Proposition \ref{sufficient} provides a sufficient condition for a non-isomorphic, 2-isogenous pair of elliptic curves to share all six BSD invariants.

Some of the pairs in Example~\ref{startingpair} satisfy this condition; in particular, the case \(m=1\) recovers Example~\ref{example}. This example is special, since the two curves also have the same Kodaira symbols at every prime and the same minimal discriminant. Applying quadratic twists then yields infinitely many further pairs.

To align the Tate–Shafarevich groups, we simultaneously trivialize their 2-torsion subgroups. By Proposition \ref{sha}, the entire Tate–Shafarevich groups therefore coincide. 

To obtain such twists, we present two complementary approaches: (i) an application of Alexander Smith’s density theorems, and (ii) an explicit construction of a twisting parameter $D$ via a $2$-descent. Using approach (ii), we also give explicit examples of pairs of elliptic curves sharing the BSD invariants, Kodaira symbols and minimal discriminants in the last section.

The logical flow of each section is roughly as follows. The parts enclosed in $\square$ are the main logic flow.

\usetikzlibrary{positioning, arrows.meta, calc, fit}
\begin{figure}[htbp]
\centering
\begin{tikzpicture}[
  node distance = 1.0cm and 1.4cm,
  every node/.style = {font=\footnotesize},
  box/.style   = {draw, rounded corners, inner sep=2pt, minimum height=4mm},
  plain/.style = {inner sep=1pt},
  edge/.style  = {-{Stealth[length=2.2mm]}, line width=0.35pt, shorten >=2pt, shorten <=2pt},
]

  \node[plain] (A) {Proposition~\ref{sha}};
  \node[box,   right=1.00cm of A] (C) {Proposition~\ref{sufficient}};
  \node[box,   right=1.00cm of C] (E) {Example~\ref{example}};
  \node[box,   below left=0.85cm and 0.55cm of E] (T) {Proposition~\ref{trueexample}};
  \node[plain, below=of A] (B) {Section~5};

  \node[box, right=3.20cm of T] (J) {Theorem~\ref{distinctj}};

  \def\RightGap{1.10cm}
  \def\Vgap{0.15cm}

  \node[plain, above=0.55cm of E] (D) {Proposition~\ref{twistsha}};
  \node[box,   right=\RightGap of D] (H1) {Theorem~\ref{thm:twists}};
  \node[plain, below=\Vgap of H1] (OR) {\textit{or}};
  \node[box,   below=\Vgap of OR]  (H2) {Proposition~\ref{pair}};
  \node[box,   right=\RightGap of H1] (F) {Theorem~\ref{main}};
  \node[plain, left=1.80cm of D] (G) {Theorem~\ref{smith}};

  \draw[edge] (A.east) -- (C.west);
  \draw[edge] (B.east) -- (C.south west); 
  \draw[edge] (C.east) -- (E.west);

  \draw[edge, dashed] (T.north east) -- (E.south west);

  \draw[edge] (T.east) -- node[above] {$+\,$Theorem~\ref{thm:twists}} (J.west);

  \draw[edge] (G.east) -- (D.west);
  \draw[edge] (D.east) -- (H1.west);

  \coordinate (ORmid) at ($ (H1.south east)!0.5!(H2.north east) $);
  \draw[edge] ($(ORmid)+(0.2cm,0)$) -- (F.west);

  \draw[edge] (E.south east) 
      .. controls ($(H2.south west)+(0,-1.3cm)$) and ($(H2.south east)+(0,-0.40cm)$) .. 
      (F.south west);

\end{tikzpicture}
\end{figure}

\section{Notation}
\subsection{Notation}
\begin{itemize}

\item  For an abelian group $A$ and a prime number $p$,
    {\setlength{\jot}{1pt}
    \[
      \begin{aligned}
        A[p]          &:= \{\,a \in A \mid pa = 0\} 
                      &&\text{($p$‑torsion subgroup)}\\
        A[p^{\infty}] &:= \{\,a \in A \mid \exists\,n\ge 1,\; p^{n}a = 0\} 
                      &&\text{($p$‑primary component)}
      \end{aligned}
    \]}

\item If \(E/\mathbb Q\) is given by a specific integral Weierstrass equation,
we denote by \(\Delta^{\mathrm{eq}}_{E}\) the discriminant of that displayed equation.
We denote by \(\Delta_{E}\) the minimal discriminant of \(E/\mathbb Q\).

\item $j(E)$: $j$-invariant of elliptic curve $E$.
\item $E^D/\mathbb{Q}$: For the Weierstrass model $E/\mathbb{Q}: y^2=x^3+ax^2+bx+c$, the quadratic twist of $E/\mathbb{Q}$ by a square-free integer $D$ is an elliptic curve defined by the equation \[E^D: y^2=x^3+aDx^2+bD^2x+cD^3.\]    

\item For an isogeny $\phi: E_1\to E_2$ between elliptic curves, $\hat{\phi}: E_2\to E_1$ denotes the dual isogeny of $\phi$.

\item For a quadratic twist $E^D$ of an elliptic curve $E$ and a $2$-isogeny $\phi: E_1 \to E_2$ defined over $\Bbb{Q}$, we also denote by $\phi$ the induced $2$-isogeny $\phi_D: E_1^D \to E_2^D$ defined over $\Bbb{Q}$. Accordingly, we abbreviate the $\phi_D$-Selmer group $\Sel^{\phi_D}(E_1^D/\mathbb{Q})$ to
$\Sel^{\phi}(E_1^D)$.

\item Let \(\phi:E_1\to E_2\) be a fixed \(2\)-isogeny, let \(D\) be a
square-free integer, and let \(v\) be a place of \(\mathbb Q\). We set
\[
W_v(\phi,D):=
\ker\left(
H^1(\mathbb Q_v,E_1^D[\phi])
\longrightarrow H^1(\mathbb Q_v,E_1^D)
\right)
=
E_2^D(\mathbb Q_v)/\phi(E_1^D(\mathbb Q_v)).
\]
Thus the \(\phi\)-Selmer group of \(E_1^D\) over \(\mathbb Q\) is
\[
\Sel^\phi(E_1^D)
=
\left\{
\xi\in H^1(\mathbb Q,E_1^D[\phi])
:\operatorname{res}_v(\xi)\in W_v(\phi,D)
\text{ for every place }v
\right\}.
\]
We also define the Tamagawa ratio by
\[
\tau_D:=
\frac{\#\Sel^\phi(E_1^D)}
{\#\Sel^{\hat\phi}(E_2^D)}.
\]

\item $\mathrm{rank}_{an}(E/\mathbb{Q})\coloneqq\mathrm{ord}_{s=1}L(E, s)$: analytic rank of $E/\mathbb{Q}$.     
\item $\mathbb{Q}(E[2])$: the minimal extension of $\mathbb{Q}$ containing every point in $E(\bar{\Bbb{Q}})[2]$.

\item $v_p(\cdot)$: For a prime number $p$, $v_p$ denotes the $p$-adic valuation.  

\item For a finite extension \(F/K\) of local or number fields, 
\(N_{F/K}:F^\times\to K^\times\) denotes the norm map.

\item For an odd prime \(p\) and an integer \(a\), \(\left(\frac{a}{p}\right)\) denotes the Legendre symbol.

\item $\Omega_E$: the real period of $E/\Bbb{Q}$.
\item  $c_{E/\mathbb Q_p}$: the Tamagawa number of $E$ at $p$.

\end{itemize}

\section{Simultaneous trivialization of \texorpdfstring{$2$}{2}-torsion subgroups of the Tate-Shafarevich groups of two elliptic curves related via a $2$-isogeny}

We begin by defining what it means for a $2$-isogeny to be balanced. The aim of
this section is to simultaneously control the Tate--Shafarevich groups of a pair
of elliptic curves connected by such an isogeny.

\begin{definition}
A $2$-isogeny $\phi: E_1\to E_2$ defined over $\Bbb{Q}$ is called balanced if it satisfies $\Bbb{Q}(E_1[2])=\Bbb{Q}(E_2[2])$.
\end{definition}

Recently, Smith's result \cite{smith} has enabled us to determine the density of the joint distribution of $\phi$-Selmer groups and $\hat{\phi}$-Selmer groups when $\phi$ is balanced. Although the explicit construction in Proposition~\ref{pair} alone would already suffice to prove our main theorem (Theorem~\ref{main}), we prove instead the more general proposition, which is non-explicit. 

We use the following proposition to simultaneously control $\Sha(E_1^D/\mathbb{Q})$ and $\Sha(E_2^D/\mathbb{Q})$.

\begin{proposition}\label{tamagawaratio}

Fix a nonzero square-free integer \(D_0\), and set
\[
S:=\{v:v\mid 2\Delta_{E_1}\infty\}.
\]

Define the Tamagawa ratio by
\[
\tau_{D} \coloneqq
 \frac{\#\,\mathrm{Sel}^{\phi}(E_1^{D})}
      {\#\,\mathrm{Sel}^{\hat{\phi}}(E_2^{D})}.\]

For any square-free $D$ satisfying \(DD_{0}\in(\Bbb{Q}_v^{\times})^{2}\) for every \(v\in S\),  \(\tau_D=\tau_{D_0}\). In particular, \(\tau_D\) is independent of the choice of \(D\).

\end{proposition}
\begin{proof}

We have 
\[
\tau_{D}
  = \prod_v \left(\frac{1}{2}\,\#\,W_v(\phi,D)\right),
\]
where $W_v(\phi,D)
  \;:=\;
  \text{Ker} \ \!\Bigl (
     H^{1}(\mathbb{Q}_v, E_1^{D}[\phi])
     \longrightarrow
     H^{1}(\mathbb{Q}_v, E_1^{D})
  \Bigr)$ \ (see Theorem 2.2 of \cite{Kl}).

When $E_1^D/\mathbb{Q}$ has good reduction at an odd prime $v$,  $W_v(\phi,D)
  \cong H^1_{\mathrm{nr}}(\Bbb{Q}_v,E_1^D[\phi])\cong H^1(\widehat{\mathbb Z},\mathbb Z/2\mathbb Z)\cong \mathbb Z/2\mathbb Z$ where $H^1_{\mathrm{nr}}(\Bbb{Q}_v,E_1^D[\phi])$ is the unramified cohomology. When $E_1/\mathbb{Q}$ has good reduction at an odd prime $v$ and $v(D)=1$, $W_v(\phi, D)\cong \mathbb{Z}/2\mathbb{Z}$ (see Section 3.1 of \cite{smith}). The remaining cases are where $E_1/\mathbb{Q}$ has bad reduction at $v$ or $v=2$ or $v=\infty$. Since $DD_0\in (\mathbb{Q}_v^{\times})^2$ for every $v\in S$, $E_1^D[\phi]\cong E_1^{D_0}[\phi]$ and $W_v(\phi, D)\cong W_v(\phi, D_0)$. Summarizing the above, $\tau_D=\tau_{D_0}$. \end{proof}

\begin{theorem}\label{tor}

Let \(E\) be an elliptic curve over \(\mathbb{Q}\).
There are at most finitely many square-free integers \(D\) for which
\(E^{D}(\mathbb{Q})\) contains a torsion point of order \(>2\).

\end{theorem}

\begin{proof}
See Proposition 1 of \cite{Mazur}.
\end{proof}

\begin{theorem}[Consequence of Theorem 1.10 of \cite{smith}]
\label{smith}
Let $\phi:E_1\to E_2$ be a balanced $2$-isogeny defined over $\Bbb{Q}$, and let $D_0$ be a nonzero integer. Set \[S\coloneqq\{v : v\mid 2\Delta_{E_1}\infty\}.\]
Let
\[
X_E(D_0,H):=
\Bigl\{
D\in\mathbb Z \;\Bigm|\;
D \text{ is square-free},\;
|D|\le H,\;
DD_0\in (\mathbb Q_v^\times)^2 \text{ for every } v\in S
\Bigr\}.
\]

For \(D\in X_E(D_0,H)\), define \(u\) by
\[
2^u=\tau_D.
\]

This quantity is independent of the choice of $D$; see Proposition \ref{tamagawaratio}. 
Let $\alpha_{r,u}$ be the limit, as $n\to\infty$, of the probability that a uniformly random $(n-u)\times n$ matrix over $\mathbb F_2$ has rank exactly $n-r$. Then, for any $r\ge 0$,
\[
\lim_{H\to\infty}
\frac{
\#\{D\in X_E(D_0,H): \dim_{\mathbb F_2}\Sel^\phi(E_1^D)=r+1\}
}{
\#X_E(D_0,H)
}
=
\alpha_{r,u}.
\]
\end{theorem}
\begin{proof}

 Let $r_\phi(E_1^D)
=
\dim_{\mathbb F_2}
\left(
\operatorname{im}\!\big(H^1(G_{\mathbb Q},E_1[\phi]) \to H^1(G_{\mathbb Q},E_1^D[2^\infty])\big)
\;\cap\;
\mathrm{Sel}_{2^\infty}(E_1^D/\Bbb{Q})
\right)$. By Theorem 1.10 of \cite{smith}, 

\[
\lim_{H\to\infty}
\frac{\#\{D \in X_E(D_0,H) : r_\phi(E_1^D)=r\}}{\#X_E(D_0,H)}
= \alpha_{r,u}.
\]

For $D$ such that $E_1^D(\Bbb{Q})[2]=E_1^D(\Bbb{Q})[4]$ and $E_2^D(\Bbb{Q})[2]=E_2^D(\Bbb{Q})[4]$, we have $\mathrm{dim}_{\Bbb{F}_2}\mathrm{Sel}^{\phi}(E_1^D)=1+r_{\phi}(E_1^D)$ as observed in Section 3.1 of \cite{smith}. By Theorem \ref{tor} we know $E_1^D(\Bbb{Q})[2]=E_1^D(\Bbb{Q})[4]$ and $E_2^D(\Bbb{Q})[2]=E_2^D(\Bbb{Q})[4]$ fail only for finitely many square-free integer $D$. Excluding these finitely many values of $D$ does not affect the density, so the result follows.
\end{proof}

\begin{proposition}\label{twistsha}
Suppose $E_i/\mathbb{Q} \ (i\in \{1,2\})$ is an elliptic curve with $E_i(\mathbb{Q})[2] \simeq \mathbb{Z}/2\mathbb{Z}$. Suppose that $E_1$ and $E_2$ are isogenous via a balanced $2$-isogeny $\phi$. Suppose that  $\mathrm{dim}_{\mathbb{F}_2} \mathrm{Sel}^{\phi}(E_1)=\mathrm{dim}_{\mathbb{F}_2} \mathrm{Sel}^{\hat{\phi}}(E_2)$.
Then, there exist infinitely many square-free integers $D>0$ such that 

\begin{itemize}
\item $\Sha(E^D_1/\mathbb{Q})[2]= \Sha(E^D_2/\mathbb{Q})[2]=0$ 
\item $\text{gcd}(D, 2\Delta_{E_1}\Delta_{E_2})=1$ and $D\equiv 1 \bmod 8$.
\item $\mathrm{rank}(E^D_i/\Bbb{Q})=0 \ (i=1,2)$ and $\mathrm{rank}_{an}(E_1^D/\Bbb{Q})\equiv 0\bmod 2$.
\end{itemize}

\end{proposition}
\begin{proof}

Let us take the integral model $E_1: y^2=x^3+ax^2+bx, E_2: y^2=x^3-2ax^2+(a^2-4b)x \ (a,b \in \Bbb{Z})$ and let $\phi: E_1\to E_2, \phi(x,y)=(\frac{y^2}{x^2}, \frac{y(b-x^2)}{x^2})$. For $D\in \Bbb{Z}$, let $s_D = \mathrm{dim}_{\mathbb{F}_2} \mathrm{Sel}^{\phi}(E^D_1)$ and $t_D = \mathrm{dim}_{\mathbb{F}_2} \mathrm{Sel}^{\hat{\phi}}(E^D_2)$. There exists an exact sequence: \[0\to E^D_2(\Bbb{Q})[\hat{\phi}]/\phi(E_1^D(\Bbb{Q})[2])\to \mathrm{Sel}^{\phi}(E^D_1)\to \mathrm{Sel}^2(E^D_1)\stackrel{\phi}{\to} \mathrm{Sel}^{\hat{\phi}}(E^D_2)\] (see [\cite{Sch}, Lemma 9.1]). Thus, 
\[\dim_{\Bbb{F}_2}\mathrm{Sel}^2(E_1^D)\le s_D+t_D-\dim_{\Bbb{F}_2}\dfrac{E^D_2(\Bbb{Q})[\hat{\phi}]}{\phi(E^D_1(\Bbb{Q})[2])},\]

\[
\dim_{\Bbb{F}_2}\mathrm{Sel}^2(E_2^D) \le t_D+s_D-\dim_{\Bbb{F}_2}\frac{E_1^D(\Bbb{Q})[\phi]}{\hat{\phi}(E_2^D(\Bbb{Q})[2])}.
\]

By Theorem $\ref{smith}$, there exist infinitely many square-free integers $D$ such that $(D, 2\Delta_{E_1}\Delta_{E_2})=1$ and $s_D = t_D = 1$ by taking $D_0=1$ and $r=0$. Indeed, with $D_0=1$, we have \[u=s_D-t_D=\mathrm{dim}_{\Bbb{F}_2}\mathrm{Sel}^{\phi}(E_1^D)-\mathrm{dim}_{\Bbb{F}_2}\mathrm{Sel}^{\hat{\phi}}(E_2^D)=\mathrm{dim}_{\Bbb{F}_2}\mathrm{Sel}^{\phi}(E_1)-\mathrm{dim}_{\Bbb{F}_2}\mathrm{Sel}^{\hat{\phi}}(E_2)=0\] by Proposition \ref{tamagawaratio}. 
Also, $D\in (\Bbb{Q}_v^{\times})^{2}$ for every $v \mid 2\Delta_{E_1}\Delta_{E_2}\infty$ implies $\text{gcd}(D, 2\Delta_{E_1}\Delta_{E_2})=1$, $D>0$ and $D\equiv 1 \bmod 8$. Hence, it remains only to show that $\alpha_{0,0}>0$. In fact,

\[
\alpha_{0,0}
=
\lim_{n\to\infty}\prod_{k=0}^{n-1}\frac{2^n-2^k}{2^n}
=
\prod_{s=1}^{\infty}(1-2^{-s})
=
0.288\ldots>0.
\]
as required (see Case 2.12 of \cite{smith2} for the formula for $\alpha_{r,u}$).

For this square-free integer $D$, we have 
\begin{equation}\label{shazero} \dim_{\mathbb{F}_2}\Sha(E^D_i/\mathbb{Q})[2]=\dim_{\mathbb{F}_2}\mathrm{Sel}^2(E^D_i/\mathbb{Q})-\dim_{\mathbb{F}_2} E^D_i(\mathbb{Q})/2E^D_i(\mathbb{Q}) \leq 1-1= 0\end{equation}
for $i \in \{1,2\}$. Additionally, since \[\#\dfrac{E^D_i(\Bbb{Q})}{2E^D_i(\Bbb{Q})}\le \#\mathrm{Sel}^2(E^D_i)\le 2\] and $\#\dfrac{E^D_i(\Bbb{Q})}{2E^D_i(\Bbb{Q})}=2^{\mathrm{rank}(E^D_i/\Bbb{Q})+1}$, we have $\mathrm{rank}(E^D_i/\Bbb{Q})=0$ for $i\in \{1,2\}$. 

Note that the condition $\mathrm{rank}_{an}(E_1^D/\Bbb{Q})=\mathrm{rank}_{an}(E_2^D/\Bbb{Q})\equiv 0\bmod 2$ follows automatically from $\Sha(E^D_1/\Bbb{Q})[2]\cong \Sha(E^D_2/\Bbb{Q})[2]=0$ and $\mathrm{rank}(E^D_1/\Bbb{Q})=\mathrm{rank}(E^D_2/\Bbb{Q})=0$. Indeed, let $\omega_E$ be a root number of $E/\Bbb{Q}$. The condition that $\mathrm{rank}(E^D_1/\Bbb{Q})=0$ and $\Sha(E^D_1/\Bbb{Q})[2]=0$ tell us that $2^{\infty}$-Selmer rank of $E^D_1/\Bbb{Q}$, that is $\mathrm{rank}(E^D_1/\Bbb{Q}) + \text{number of copies of } \Bbb{Q}_2/\Bbb{Z}_2 \text{ in } \Sha(E^D_1/\Bbb{Q})$, is zero. By \cite{monsky}, Theorem 1.5, this implies $\omega(E^D_1/\Bbb{Q})=1$, thus $\mathrm{rank}_{an}(E^D_1/\Bbb{Q})$ is even. 
\end{proof}

\section{Kodaira symbol, Tamagawa number and real period}

In this section, we recall the results on Kodaira symbols, Tamagawa numbers, and real periods.

\begin{theorem}\label{dok}
Let $E_1/\Bbb{Q}$ and $E_2/\Bbb{Q}$ be elliptic curves that are isogenous by a $2$-isogeny over $\Bbb{Q}$. Assume
$E_1$ has good reduction at $2$.  The following are equivalent:
\begin{enumerate}
\item $E_1$ and $E_2$ have the same Kodaira symbol at every prime $\ell$;
\item $\Delta_{E_1}=\pm\Delta_{E_2}$.
\end{enumerate}
\end{theorem}

\begin{proof}
($(1)\Rightarrow (2)$) Ogg's formula states that \[v_\ell(\Delta_E) = f_{\ell, E} + m_{\ell, E} - 1,\] where $f_{\ell,E}$ is the conductor exponent of $E$, $m_{\ell,E}$ is the number of components in the special fibre of the minimal regular model of $E$, and $\Delta_E$ is the minimal discriminant of $E$.
(see \cite{silad}, Chapter VI, §11). Since $E_1$ and $E_2$ are isogenous, the conductor exponent $f_\ell$ is an isogeny invariant, so $f_{\ell,E_1} = f_{\ell,E_2}$ for all primes $\ell$. Since $E_1$ and $E_2$ have the same Kodaira symbols at all primes $\ell$, we have $m_{\ell,E_1} = m_{\ell,E_2}$ for all primes $\ell$. Therefore, $v_\ell(\Delta_{E_1}) = v_\ell(\Delta_{E_2})$ for all primes $\ell$. This implies $\Delta_{E_1} = \pm \Delta_{E_2}$.

($(1) \Leftarrow (2)$) Suppose $\Delta_{E_1} = \pm \Delta_{E_2}$. Then for all primes $\ell$, we have $v_\ell(\Delta_{E_1}) = v_\ell(\Delta_{E_2})$. By Ogg's formula, this means $f_{\ell,E_1} + m_{\ell,E_1} - 1 = f_{\ell,E_2} + m_{\ell,E_2} - 1$. Since $E_1$ and $E_2$ are isogenous, $f_{\ell,E_1} = f_{\ell,E_2}$, which implies $m_{\ell,E_1} = m_{\ell,E_2}$. For $\ell=2$, since $E_1$ has good reduction at $2$ by assumption and good reduction is preserved under isogeny, $E_2$ also has good reduction at $2$. Hence both curves have Kodaira symbol \(I_0\) at \(2\). Now let \(\ell>2\). By Lemma 2.11 of \cite{twin}, \(E_1\) has potentially good reduction at \(\ell\). By Theorem 5.4 (1) of \cite{dokchitser}, the Kodaira symbols are the same. 

\end{proof}

\begin{theorem}\label{kodaira}
Suppose $E_1/\mathbb{Q}$ and $E_2/\mathbb{Q}$ are elliptic curves such that there exists a $2$-isogeny over $\Bbb{Q}$ between them. Let $\ell\ge 3$ be a prime. If $\Delta_{E_1}=\Delta_{E_2}$, then $c_{E_1/\mathbb{Q}_\ell} = c_{E_2/\mathbb{Q}_\ell}$.
\end{theorem}

\begin{proof}
 Since $\Delta_{E_1}=\Delta_{E_2}$, $E_1$ and $E_2$ have potentially good reduction at $\ell$ by Lemma 2.11 of \cite{twin}.
When the Kodaira symbol at $\ell$ is not $I^*_0$, then $c_{E_1/\mathbb{Q}_\ell} = c_{E_2/\mathbb{Q}_\ell}$ (see the proof of Theorem 6.1, \cite{dokchitser}). When the Kodaira symbol at $\ell$ is $I^*_0$, the result follows since $1=\frac{\Delta_{E_1}}{\Delta_{E_2}} \in N_{F/\Bbb{Q}_\ell}(F^{\times})$ where $F$ is a quadratic extension of $\Bbb{Q}_\ell$ over which $E_1$ has good reduction (see \cite{dokchitser}, Table 1 and Theorem 6.1). 
  
  \end{proof}

\begin{theorem}\label{period}
Let $p$ be a prime number. Let $E_1/\Bbb{Q}$ and $E_2/\Bbb{Q}$ be elliptic curves that are isogenous by a $p$-isogeny over $\Bbb{Q}$. Assume that $E_1$ and $E_2$ have the same Tamagawa numbers at all primes and $\mathrm{rank}_{an}(E_1/\Bbb{Q})\equiv 0 \bmod 2$.  Then, $\Omega_{E_1}=\Omega_{E_2}$.

If $\text{BSD}(E_1/\mathbb{Q}) = \text{BSD}(E_2/\mathbb{Q})$ and $E_1$ and $E_2$ are isogenous by $p$-isogeny over $\Bbb{Q}$, then $E_1/\mathbb{Q}$ and $E_2/\Bbb{Q}$ have even analytic rank.
\end{theorem}

\begin{proof}
By Theorem 8.2 of \cite{dokchitser}, \[\Omega_{E_1}=\Omega_{E_2} \iff \sum_\ell \text{ord}_p (\frac{c_{E_1/\mathbb{Q}_\ell}}{c_{E_2/\mathbb{Q}_\ell}})\equiv \text{ord}_{s=1} L(E_1,s) \bmod 2.\] 
Since Tamagawa numbers at all primes are the same, $\sum_l \text{ord}_p (\frac{c_{E_1/\mathbb{Q}_\ell}}{c_{E_2/\mathbb{Q}_\ell}})=0$ holds. Thus, the equality $\Omega_{E_1}=\Omega_{E_2}$ is equivalent to the congruence $\text{ord}_{s=1} L(E_1,s)\equiv 0 \bmod 2$, which is true.

Now suppose that $\mathrm{BSD}(E_1/\mathbb Q)=\mathrm{BSD}(E_2/\mathbb Q)$.
Then, in particular, the Tamagawa numbers agree at all primes and
\(\Omega_{E_1}=\Omega_{E_2}\). The equivalence
proved above implies $\operatorname{ord}_{s=1}L(E_1,s)\equiv 0 \bmod 2$. Thus both analytic ranks are even.
\end{proof}

\section{Infinitely many pairs of elliptic curves sharing the same BSD data}

In this section, we show that there are infinitely many BSD twins, and that infinitely many twins remain even after imposing equality of the Kodaira symbols and the minimal discriminants. We first observe that a pair connected by a $p$-isogeny over $\mathbb{Q}$ has isomorphic Mordell--Weil groups and Tate--Shafarevich groups, provided that their $p$-parts agree. Here we do not assume the finiteness of the Tate--Shafarevich groups.

\begin{lemma}\label{functor}
Let $F$ be a covariant functor from the category of elliptic curves over $\Bbb{Q}$ to the category of abelian groups such that $F([n]) = [n]$ for all positive integers $n$ where $[n]$ is the multiplication by $n$ map. Let $p$ be a prime, and suppose $E_1$ and $E_2$ are connected by an isogeny $\phi$ of degree $p$. Assume that for any elliptic curve $E$, $F(E)$ is a torsion group. If $F(E_1)[p^{\infty}] \cong F(E_2)[p^{\infty}]$ holds, then $F(E_1) \cong F(E_2)$ holds.
\end{lemma}

\begin{proof} Since $\phi \circ \hat{\phi} = [p]$ and $\hat{\phi} \circ \phi = [p]$, we have $F(\phi) \circ F(\hat{\phi}) = [p]$ and $F(\hat{\phi}) \circ F(\phi) = [p]$. For a prime $\ell\neq p$, since $\gcd(p,\ell) = 1$, $F([p]) = [p]: F(E_1)[\ell^{\infty}] \to F(E_1)[\ell^{\infty}]$ is bijective, it follows that $F(\phi): F(E_1)[\ell^{\infty}]\to F(E_2)[\ell^{\infty}]$ is an isomorphism. Since $F(E_1)$ and $F(E_2)$ are torsion groups, we have
\[F(E_1) \cong \bigoplus_{q: \text{prime}} F(E_1)[q^{\infty}] \cong \bigoplus_{q: \text{prime}} F(E_2)[q^{\infty}]\cong F(E_2).\]
\end{proof}

\begin{proposition}\label{sha}
Let $p$ be a prime. Let $E_1$ and $E_2$ be elliptic curves over $\mathbb{Q}$ that are
$p$-isogenous over $\mathbb{Q}$.

(1) Suppose $E_1(\mathbb{Q})[p^{\infty}] \cong E_2(\mathbb{Q})[p^{\infty}]$ holds. Then $E_1(\mathbb{Q}) \cong E_2(\mathbb{Q})$.

(2) Let $p$ be a prime and suppose $\Sha(E_1/\mathbb{Q})[p^{\infty}] \cong \Sha(E_2/\mathbb{Q})[p^{\infty}]$ holds. Then, without assuming the finiteness of the Tate--Shafarevich group, $\Sha(E_1/\mathbb{Q}) \cong \Sha(E_2/\mathbb{Q})$ holds.
\end{proposition}

\begin{proof}(1) The rank is invariant under isogeny. By Lemma \ref{functor}, the torsion parts are also isomorphic. By the Mordell--Weil Theorem, $E_1(\Bbb{Q})\cong E_2(\Bbb{Q})$.\par

(2) For an isogeny \(f:A\to B\) of elliptic curves over \(\mathbb{Q}\),
functoriality of Galois cohomology induces homomorphisms
\[
f_*:H^1(\mathbb{Q},A)\to H^1(\mathbb{Q},B)
\quad\text{and}\quad
f_{v,*}:H^1(\mathbb{Q}_v,A)\to H^1(\mathbb{Q}_v,B)
\]
for every place \(v\) of \(\mathbb{Q}\). These maps are compatible with the
restriction maps. Hence \(f_*\) sends locally trivial classes to locally
trivial classes and induces a homomorphism
\[
f_*:\Sha(A/\mathbb{Q})\to \Sha(B/\mathbb{Q}).
\]
Thus \(A\mapsto \Sha(A/\mathbb{Q})\) is functorial in \(A\). Moreover, for
every integer \(n\ge1\), the morphism \([n]:A\to A\) induces multiplication
by \(n\) on \(\Sha(A/\mathbb{Q})\).

Any element $[C]$ of the Tate--Shafarevich group vanishes when lifted to some degree $n$ extension, we have $n \cdot [C] = 0$. Therefore, the Tate-Shafarevich group is a torsion group. By Lemma \ref{functor}, we have $\Sha(E_1/\mathbb{Q}) \cong \Sha(E_2/\mathbb{Q})$.
\end{proof}

\begin{proposition}\label{sufficient}
Let $E_1/\mathbb{Q}$ and $E_2/\mathbb{Q}$ be elliptic curves that are isogenous via a $2$-isogeny over $\mathbb{Q}$. Assume that
\begin{enumerate}
\item $E_1(\mathbb{Q})[2^\infty]\cong E_2(\mathbb{Q})[2^\infty]$;
\item $\Sha(E_1/\mathbb{Q})[2^\infty]\cong \Sha(E_2/\mathbb{Q})[2^\infty]$, $\operatorname{rank}_{\mathrm{an}}(E_1/\mathbb{Q})\equiv 0 \bmod{2}$, and $\operatorname{rank}(E_1/\mathbb{Q})=0$.

\item Suppose further that $E_1$ has good reduction at $2$, and that one of the following holds:
\begin{enumerate}
\item[(a)] $c_{E_1/\mathbb{Q}_p}=c_{E_2/\mathbb{Q}_p}$ for every prime $p$;
\item[(b)] $\Delta_{E_1}=\Delta_{E_2}$.
\end{enumerate}
\end{enumerate}
Then $\mathrm{BSD}(E_1/\mathbb{Q})=\mathrm{BSD}(E_2/\mathbb{Q})$. Moreover, in case \textup{(b)}, the curves $E_1$ and $E_2$ have the same Kodaira symbols at every prime.

\end{proposition}
\begin{proof}
Since $E_1$ and $E_2$ are isogenous over $\Bbb{Q}$, $L(E_1,s)=L(E_2,s)$. By Proposition \ref{sha}, we have $E_1(\Bbb{Q})\cong E_2(\Bbb{Q})$ and $\Sha(E_1/\Bbb{Q})\cong \Sha(E_2/\Bbb{Q})$.  Because $\mathrm{rank}(E_i/\Bbb{Q})=0$, $\mathrm{Reg}(E_i/\Bbb{Q})=1$ for $i\in \{1,2\}$. Thus, in case \textup{(a)}, $\Omega_{E_1}=\Omega_{E_2}$ follows from Theorem \ref{period}. Thus, we have $\mathrm{BSD}(E_1/\mathbb{Q})=\mathrm{BSD}(E_2/\mathbb{Q})$. In case \textup{(b)}, by Theorem \ref{kodaira}, Tamagawa numbers at $\ell \ge 3$ are the same because $\Delta_{E_1}=\Delta_{E_2}$. Thus, we have $\Omega_{E_1}=\Omega_{E_2}$ follows from Theorem \ref{period}. By Theorem \ref{dok}, the Kodaira symbols are the same.

\end{proof}

\begin{example}\label{example}
An example of a pair $(E_1,E_2)$ satisfying conditions 1,2, and 3 of Proposition~\ref{sufficient} is
\[
\left(
y^2=x^3+25350x^2+2471625x,\ 
y^2=x^3-50700x^2+632736000x
\right).
\]
The LMFDB labels of $E_1$ and $E_2$
are 38025.ck1 and 38025.ck2, respectively. Indeed, for $(1)$, $E_1(\Bbb{Q})[2^{\infty}]\cong E_2(\Bbb{Q})[2^{\infty}]$. For $(2)$, a Sage computation shows that, for $i=1,2$, $\Sel^2(E_i) \cong \mathbb{Z}/2\mathbb{Z},
\mathrm{rank}(E_i/\mathbb{Q})=0,
E_i(\mathbb{Q})[2]\cong \mathbb{Z}/2\mathbb{Z}$. Hence $E_i(\mathbb{Q})/2E_i(\mathbb{Q})\cong \mathbb{Z}/2\mathbb{Z}$,
and therefore $\Sha(E_1/\mathbb{Q})[2]=\Sha(E_2/\mathbb{Q})[2]=0$
and thus $\Sha(E_1/\mathbb{Q})[2^{\infty}]\cong \Sha(E_2/\mathbb{Q})[2^{\infty}]=0$. Also, a Sage computation shows that $\operatorname{rank}_{\mathrm{an}}(E_1/\mathbb Q)=0$. For $(3)$, $\Delta_{E_1}=\Delta_{E_2}=3^6\cdot 5^9 \cdot 13^9$. 
This is the unique pair of 2-isogenous elliptic curves that shares the same minimal discriminants, up to twists. Indeed, the pair \((38025.\mathrm{ck}1,\,38025.\mathrm{ck}2)\) is the quadratic
twist by \(-195\) of the pair \((4225.\mathrm{h}1,\,4225.\mathrm{h}2)\) of Theorem \ref{discriminanttwins}.
We use this twist because \(4225.\mathrm{h}1\) has odd analytic rank; by
Theorem~\ref{period}, the original pair therefore does not have the same real
period and hence is not a BSD twin.\end{example}

The following is a generalization of Example \ref{example}. Note that the pair \((38025.\mathrm{ck}1, 38025.\mathrm{ck}2)\) in Example \ref{example} coincides with the case \(m=1\).

\begin{lemma}\label{discriminantratio}
Let \(m\) be a prime such that \(m\nmid 2\cdot 3\cdot 5\cdot 13\), and set $n\coloneqq m^2+64$. Let $E_m:\ y^2=x^3+390n x^2+195^2m^2n\,x$
and $E'_m:\ y^2=x^3-195n x^2+16\cdot195^2n\,x$. Then
\[
\frac{\Delta_{E_m}}{\Delta_{E'_m}}=m^2.
\]

\end{lemma}

\begin{proof}
The discriminant of a Weierstrass equation of the form $y^2=x^3+Ax^2+Bx$ is $\Delta=16B^2(A^2-4B)$. Applying this formula to the displayed equations of \(E_m\) and \(E'_m\), we obtain $\Delta^{\mathrm{eq}}_{E_m}=2^{12}195^6m^4n^3$ and $\Delta^{\mathrm{eq}}_{E'_m}=2^{12}195^6m^2n^3$.
Hence
\[
\frac{\Delta^{\mathrm{eq}}_{E_m}}{\Delta^{\mathrm{eq}}_{E'_m}}=m^2.
\]

Fix a prime \(p\). It is enough to show that, for every prime \(p\), the same power
of \(p^{12}\) is removed from the two displayed discriminants when passing to
minimal models.

At \(p=2\), the change of variables $x=4X$, $y=8Y+4X$ gives
\[
Y^2+XY=X^3-\frac{195m^2+12481}{4}X^2+195^2n\,X,
\]
whose discriminant is $195^6m^2n^3$, which is odd. Thus $E'_m$ has good reduction at $2$. Therefore $v_2(\Delta_{E_m})=v_2(\Delta_{E'_m})=0$.
Since $v_2(\Delta^{\mathrm{eq}}_{E_m})
=
v_2(\Delta^{\mathrm{eq}}_{E'_m})
=12$,
the same factor \(2^{12}\) is removed from both displayed discriminants.

Next consider \(p=m\). Since \(m\nmid 195\) and $n=m^2+64\equiv 64\not\equiv 0 \bmod m$,
we have $v_m(\Delta^{\mathrm{eq}}_{E_m})=4$
and $v_m(\Delta^{\mathrm{eq}}_{E'_m})=2$. Since both valuations are \(<12\), both displayed equations are \(m\)-minimal.

It remains to consider an odd prime \(p\neq m\). Set
\[
k_p:=\left\lfloor \frac{v_p(\Delta^{\mathrm{eq}}_{E_m})}{12}\right\rfloor
=
\left\lfloor \frac{v_p(\Delta^{\mathrm{eq}}_{E'_m})}{12}\right\rfloor
=
\left\lfloor \frac{v_p(n)+2v_p(195)}{4}\right\rfloor .
\]
We claim that the change of variables $x=p^{2k_p}X, y=p^{3k_p}Y$
gives integral Weierstrass equations for both curves. Indeed, after this change, the equation $y^2=x^3+Ax^2+Bx$ becomes $Y^2=X^3+p^{-2k_p}AX^2+p^{-4k_p}BX$.
Thus, we have 
\[
v_p(390n)=v_p(n)+v_p(195)\ge\frac{v_p(n)+2v_p(195)}{2}\ge 2k_p,\]

\[ v_p(195^2m^2n)=2v_p(195)+v_p(n)\ge 4k_p.
\]

Hence the transformed equations are integral for both \(E_m\) and \(E'_m\).

This change decreases both discriminant valuations by \(12k_p\), and the remaining
valuations are \(<12\). Hence the transformed equations are \(p\)-minimal, so exactly
the same factor \(p^{12k_p}\) is removed from both displayed discriminants.

Therefore
\[
\frac{\Delta_{E_m}}{\Delta_{E'_m}}
=
\frac{\Delta^{\mathrm{eq}}_{E_m}}{\Delta^{\mathrm{eq}}_{E'_m}}
=m^2.\] \end{proof}

\begin{proposition}\label{trueexample}

Let \(m\) be either \(1\) or a prime satisfying
$\left(\frac{390}{m}\right)=-1$, and set \(n\coloneqq m^2+64\). Let $E_m$ and $E'_m$ be the elliptic curves defined by
\[
E_m:\ y^2=x^3+390n x^2+195^2m^2n\,x,
\qquad
E'_m:\ y^2=x^3-195n x^2+16\cdot 195^2n\,x.
\]
Then $j(E_m)\neq j(E'_m)$, and as $m$ varies, the pairs $(j(E_m),j(E'_m))$ are pairwise distinct.

The $2$-isogeny
\[
\hat{\phi}:E'_m\longrightarrow E_m,\qquad
(x,y)\longmapsto
\left(
\frac{y^2}{x^2},
\frac{y\left(16\cdot 195^2 n-x^2\right)}{x^2}
\right)
\]
is balanced.

Moreover, we have $c_{E_m/\Bbb{Q}_p}=c_{E'_m/\Bbb{Q}_p}\qquad\text{for every prime }p$.
\end{proposition}

\begin{proof}

The case \(m=1\) has already been treated in Example~\ref{example}. We may therefore assume that \(m\neq 1\).

The direct calculation shows 
\[
j(E_m)=\frac{(m^2+256)^3}{m^4},
\qquad
j(E'_m)=\frac{(m^2+16)^3}{m^2},
\]
and
\[
j(E_m)-j(E'_m)
=
-\frac{(m-8)(m+8)(m^2+64)(m^2-9m+64)(m^2+9m+64)}{m^4}.
\]
In particular, $j(E_m)\neq j(E'_m)$ for every such $m$. 

The pairs of \(j\)-invariants are pairwise distinct as \(m\) varies. Indeed,
\(j(E'_m)=\dfrac{(m^2+16)^3}{m^2}\) is in lowest terms, since
\(\gcd(m,m^2+16)=1\). Hence its denominator determines \(m\).

The $2$-division polynomials are $x\bigl(x^2+390nx+195^2m^2n\bigr)
\quad\text{and}\quad
x\bigl(x^2-195nx+16\cdot 195^2n\bigr)$.
Their discriminants are $256\cdot 195^2n$ and $195^2m^2n$, respectively, so
\[
\mathbb Q(E_m[2])=\mathbb Q(E'_m[2])=\mathbb Q(\sqrt n).
\]
Hence the isogeny is balanced.

For $p=2$, $E'_m$ has good reduction at $2$, so $c_{E'_m/\Bbb{Q}_2}=1$ by the proof of Lemma \ref{discriminantratio}.

We have  
\[
\frac{\Delta_{E_m}}{\Delta_{E'_m}}=m^2\in\mathbb Q^{\times 2}.
\] by Lemma \ref{discriminantratio}.

We now compare Tamagawa numbers.

Since $E_m$ and $E'_m$ are isogenous over $\Bbb{Q}$, $E_m$ also has good reduction at $2$, so $c_{E_m/\Bbb{Q}_2}=1$.

Next let \(q\) be an odd prime dividing \(195n\) with \(q\neq m\).
Since \(q\nmid m\), both
\[
j(E_m)=\frac{(m^2+256)^3}{m^4},
\qquad
j(E'_m)=\frac{(m^2+16)^3}{m^2}
\]
are \(q\)-integral, so both curves have potentially good reduction at \(q\). Since $\frac{\Delta_{E_m}}{\Delta_{E'_m}}=m^2 \in N_{F/\Bbb{Q}_{q}}(F^{\times})$, the same argument as in the proof of Theorem \ref{kodaira} yields
\[
c_{E_m/\Bbb{Q}_q}=c_{E'_m/\Bbb{Q}_q}.
\]

Finally, assume $m\neq 5,13$ and set $q=m$. Since $q\nmid 195$ and $q\nmid n$, reducing modulo $q$ gives
\[
E'_m:\ y^2=x^3-195\cdot 64\,x^2+16\cdot 195^2\cdot 64\,x
      =x(x-6240)^2,
\]
and
\[
E_m:\ y^2=x^3+390\cdot 64\,x^2
      =x^2(x+24960).
\]
Thus both curves have multiplicative reduction at $q$, and the reduction types are $I_4$ and $I_2$, respectively. 

The fact that the slope of the tangent line does not lie in \(\mathbb{F}_m\), equivalently, that the reduction is non-split multiplicative, is equivalent to
\[
\left(\frac{390}{m}\right)=-1.
\]

Hence $c_{E_m/\Bbb{Q}_m}=2, c_{E'_m/\Bbb{Q}_m}=2$.
Combining all cases, we obtain
\[
c_{E_m/\Bbb{Q}_p}=c_{E'_m/\Bbb{Q}_p} \ \text{for every prime }p.
\]

\FloatBarrier

\begin{table}[htbp]
\centering
\caption{Reduction types of $E_m$ and $E'_m$ }
\label{tab:reduction-types-Em}
\small
\setlength{\tabcolsep}{4pt}
\renewcommand{\arraystretch}{1.08}
\begin{tabular}{c|c|c}
\hline
prime $p$ & $E_m$ & $E'_m$ \\
\hline
$2$ & good & good \\
$q\mid 195n,\ q\neq m$ & potentially good & potentially good \\
$p=m$ & multiplicative of type $I_4$ & multiplicative of type $I_2$ \\
$p\nmid 195mn$ & good & good \\
\hline
\end{tabular}

\vspace{1mm}
\parbox{0.92\linewidth}{\footnotesize
Here, $\left(\frac{390}{m}\right)=-1$, and $n=m^2+64$.
Moreover, $c_{E_m/\Bbb{Q}_2}=c_{{E'_m/\Bbb{Q}_2}}=1$, one has $c_{E_m/\Bbb{Q}_q}=c_{E'_m/\Bbb{Q}_q}$ for every prime
$q\mid 195n$ with $q\neq m$, and at $p=m$ both curves have non-split multiplicative
reduction, so $c_{E_m/\Bbb{Q}_m}=c_{E'_m/\Bbb{Q}_m}=2$.
}
\end{table}

\FloatBarrier

\end{proof}

\begin{theorem}\label{thm:twists}
Suppose $E_i/\mathbb{Q} \ (i\in \{1,2\})$ is an elliptic curve with $E_i(\mathbb{Q})[2] \simeq \mathbb{Z}/2\mathbb{Z}$. Suppose that $E_1$ and $E_2$ are isogenous via a balanced $2$-isogeny
$\phi\colon E_1 \to E_2$ defined over $\mathbb{Q}$. Suppose that  $\# \mathrm{Sel}^{\phi}(E_1)=\# \mathrm{Sel}^{\hat{\phi}}(E_2)$. Also, suppose that $E_1$ has good reduction at $2$ and that one of the following holds:
\begin{enumerate}
\item[(a)] $c_{E_1/\mathbb{Q}_p}=c_{E_2/\mathbb{Q}_p}$ for every prime $p$;
\item[(b)] $\Delta_{E_1}=\Delta_{E_2}$.
\end{enumerate}

Then, there exist infinitely many square-free integers $D$ such that $\text{BSD}(E^D_1/\mathbb{Q})=\text{BSD}(E^D_2/\mathbb{Q})$. Moreover, in case \textup{(b)}, the curves $E^D_1$ and $E^D_2$ have the same Kodaira symbols at every prime.
\end{theorem}
\begin{proof}

By Proposition~\ref{sufficient}, it suffices to show that there are infinitely many square-free integers \(D\) such that \(E_1^D\) and \(E_2^D\) are $2$-isogenous over \(\mathbb{Q}\) and satisfy conditions 1, 2, and 3 of that proposition.

For square-free integers $D$, $E^D_1$ and $E^D_2$ are $2$-isogenous over $\Bbb{Q}$. Choose \(D\) among the infinitely many square-free integers given by Proposition \ref{twistsha}. We show that, for all but finitely many such \(D\), the following conditions hold.

\begin{enumerate}
    \item $E^D_1(\mathbb{Q})[2^{\infty}]\cong E^D_2(\mathbb{Q})[2^{\infty}]$.
    \item $\Sha(E^D_1/\mathbb{Q})[2^{\infty}]\cong \Sha(E^D_2/\mathbb{Q})[2^{\infty}]=0$, \\[-1pt]
$\operatorname{rank}_{an}(E^D_1/\mathbb{Q})\equiv 0 \bmod 2$ and $\mathrm{rank}(E^D_1/\Bbb{Q})=0$.
    \item $E^D_1$ has good reduction at $2$, $c_{E^D_1/\mathbb{Q}_p}=c_{E^D_2/\mathbb{Q}_p}$ for every prime $p$ in case \textup{(a)}    
    and $\Delta_{E^D_1}=\Delta_{E^D_2}$ in case \textup{(b)}.
\end{enumerate}

For condition $1$, $E_1^D(\Bbb{Q})[2]\cong E_2^D(\Bbb{Q})[2]$. By Theorem \ref{tor}, there are finitely many $D$ such that $E^D_1(\Bbb{Q})_{\text{tor}}$ or $E^D_2(\mathbb{Q})_{\text{tor}}$ acquire points of order greater than 2. Thus, for almost all square-free integers $D$, we have $E^D_1(\Bbb{Q})[2^{\infty}]\cong E^D_2(\Bbb{Q})[2^{\infty}]$. 

For condition $2$,  $\# \mathrm{Sel}^{\phi}(E^D_1)=\# \mathrm{Sel}^{\hat{\phi}}(E^D_2)$ by Proposition \ref{tamagawaratio}. By Proposition \ref{twistsha}, infinitely many square-free integers $D$ satisfy $\Sha(E^D_1/\Bbb{Q})[2]=\Sha(E^D_2/\Bbb{Q})[2]=0$, $(D, 2\Delta_{E_1}\Delta_{E_2})=1$, $D\equiv 1 \bmod 8$, $\mathrm{rank}(E_i^D/\Bbb{Q})=0 \ (i=1,2)$ and $\mathrm{rank}_{an}(E^D_1/\Bbb{Q})\equiv 0 \bmod 2$. For such $D$, by Proposition \ref{sha} (2), $\Sha(E^D_1/\Bbb{Q})\cong \Sha(E^D_2/\Bbb{Q})$.

For condition \(3\), in case \textup{(a)}, for any prime \(p \nmid 2\Delta_{E_1}D\), the curve \(E_1^D\) has good reduction at \(p\), since \((D,2\Delta_{E_1})=1\).

For any prime \(p\mid 2\Delta_{E_1}\Delta_{E_2}\), our choice of \(D\) implies that \(D\) is a square in \(\mathbb{Q}_p^\times\). Hence \(E_i^D\) is isomorphic to \(E_i\) over \(\mathbb{Q}_p\) for \(i=1,2\). Therefore
\[
c_{E_1^D/\mathbb{Q}_p}
=
c_{E_1/\mathbb{Q}_p}
=
c_{E_2/\mathbb{Q}_p}
=
c_{E_2^D/\mathbb{Q}_p}.
\]

For \(p=2\), both \(E_1\) and \(E_1^D\) have good reduction at \(2\), since \(D \equiv 1 \bmod 8\). For any prime \(p \mid D\), the curve \(E_1/\mathbb{Q}\) has good reduction at the odd prime \(p\), and \(v_p(D)=1\). Hence
\[
W_p(\phi,D)\cong \mathbb{Z}/2\mathbb{Z}
\]
(see Section 3.1 of \cite{smith}). For any prime $p\neq 2$, Lemma~3.8 of \cite{Sc}
(or, alternatively, Lemma~2.1(ii) of \cite{zy}), we have $\frac{c_{E^D_2/\Bbb{Q}_p}}{c_{E^D_1/\Bbb{Q}_p}}=\frac{1}{2} \#W_p(\phi, D)=1$. Thus, for all prime numbers \(p\), we have
\[
c_{E_1^D/\mathbb{Q}_p}
=
c_{E_2^D/\mathbb{Q}_p}.
\]

In case \textup{(b)}, since $(D, 2\Delta_{E_1}\Delta_{E_2})=1$ and $D\equiv 1 \bmod 8$, $E^D_1$ has good reduction at $2$ by Proposition 2.4 (2)(a) \cite{pal}. Since $\Delta_{E_1}=\Delta_{E_2}$, $j(E^D_i)=j(E_i)\in \Bbb{Z}$ for $i=1,2$ by Lemma 2.11 of \cite{twin}. Therefore, $E^D_1$ and $E^D_2$ have the same Kodaira symbols at every prime by Theorem 5.4 (1) of \cite{dokchitser}. Thus, by Theorem \ref{dok} and since quadratic twisting preserves the sign of the discriminant, we have
\(\Delta_{E_1^D}=\Delta_{E_2^D}\).

\end{proof}

\begin{theorem}\label{j}

Take $E_m,E'_m$ as in Proposition \ref{trueexample}. If $\mathrm{rank}_{an}(E_m/\mathbb{Q})\equiv 0 \bmod 2$ (for example, $m=7,19,23,37,43, 53$), then there exist infinitely many square-free integers $D$ such that the pair $(E^D_m,E'^D_m)$ satisfies
\[
j(E^D_m)\neq j(E'^D_m)
\quad\text{and}\quad
\operatorname{BSD}(E^D_m/\mathbb{Q})=\operatorname{BSD}(E'^D_m/\mathbb{Q}),
\]
and the pairs $(j(E^D_m),j(E'^D_m))=(j(E_m),j(E_m'))$ are pairwise distinct if we vary $m$.
\end{theorem}
\begin{proof}
Now $E_m$ and $E'_m$ are connected by a balanced $2$-isogeny over $\Bbb{Q}$, and $E_m(\Bbb{Q})[2]\cong E'_m(\Bbb{Q})[2]\cong \Bbb{Z}/2\Bbb{Z}$ since $n$ is not a square.

Moreover, they have the same Tamagawa numbers by Proposition ~\ref{trueexample}.

By Theorem~1 and Remark~1 of \cite{Kloosterman}, we have

\[
\frac{\#\Sel^{\phi}(E_m)}{\#\Sel^{\hat{\phi}}(E'_m)}
=
\frac{
\#E_m(\mathbb{Q})[\phi]\,\Omega_{E'_m}\,\prod_{p} c_{E'_m/\mathbb{Q}_p}
}{
\#E'_m(\mathbb{Q})[\hat{\phi}]\,\Omega_{E_m}\,\prod_{p} c_{E_m/\mathbb{Q}_p}
}.
\]

Since $\mathrm{rank}_{\mathrm{an}}(E_m/\mathbb{Q})\equiv 0\bmod 2 $, we have $\Omega_{E_m}=\Omega_{E'_m}$ by Theorem~\ref{period}. By Proposition \ref{trueexample}, $c_{E_m/\Bbb{Q}_p}=c_{E'_m/\mathbb{Q}_p}$. This proves that $\tau_1=1$. By applying Theorem~\ref{thm:twists}, we obtain infinitely many square-free integers $D$ such that
\[
\mathrm{BSD}(E_m^D/\mathbb{Q})=\mathrm{BSD}(E_m'^D/\mathbb{Q}).
\]
Thus, the pairs $(E_m^D,E_m'^D)$ give the desired examples. 
If we vary $m$, Proposition~\ref{trueexample} implies that these pairs of $j$-invariants are pairwise distinct.
\end{proof}

In the case \(m=1\), the pair \((E_m,E'_m)\) cannot be distinguished even by the finer local data, namely the Kodaira symbols and the minimal discriminants.

\begin{theorem}\label{maintheorem}
There exist infinitely many pairs $(E_1,E_2)$ of non-isomorphic elliptic curves over $\mathbb{Q}$ such that $j(E_1)\neq j(E_2)$, $\text{BSD}(E_1/\mathbb{Q})=\text{BSD}(E_2/\mathbb{Q})$, and the Kodaira symbols at every prime and the minimal discriminants are the same.
\end{theorem}
\begin{proof}
The pairs in Example \ref{example} satisfy the hypothesis of Proposition \ref{sufficient} and Theorem \ref{thm:twists}. By Theorem \ref{thm:twists}, infinitely many pairs $(E^D_1, E^D_2)$ are the pairs we have sought.
\end{proof}

\

\FloatBarrier

\section{Explicit infinite families of pairs}

In this section, we explicitly provide infinitely many integers $D$ such that, for the pair \[(38025.\mathrm{ck}1,\,38025.\mathrm{ck}2)\]---the case $m=1$ of Example \ref{startingpair}---the 2-primary parts of the Tate--Shafarevich groups are simultaneously trivial.

\subsection{Simultaneous trivialization of the 2-part of the Tate--Shafarevich groups}

The following lemma enables us to control $\Sha(E_1^D/\mathbb{Q})[2]$ and $\Sha(E_2^D/\mathbb{Q})[2]$ simultaneously. Indeed, by this lemma, once we prove that $\mathrm{Sel}^{\phi}(E_1^D)=\mathbb{Z}/2\mathbb{Z}$, it follows that $\mathrm{Sel}^{\hat{\phi}}(E_2^D)=\mathbb{Z}/2\mathbb{Z}$. Then, by the same argument as in Proposition \ref{twistsha}, we obtain $\Sha(E_1^D/\mathbb{Q})[2]=\Sha(E_2^D/\mathbb{Q})[2]=0$.

\begin{lemma}\label{taud=1}
Let
\[
E_1: y^2=x^3+25350x^2+2471625x
\quad\text{and}\quad
E_2: y^2=x^3-50700x^2+632736000x
\]
be the elliptic curves with LMFDB labels \(38025.\mathrm{ck}1\) and
\(38025.\mathrm{ck}2\), respectively.  They are connected by the $2$-isogeny
\[
\phi:E_1\longrightarrow E_2,\qquad
(x,y)\longmapsto
\left(\frac{y^2}{x^2},\frac{y(2471625-x^2)}{x^2}\right)
\]
defined over \(\mathbb Q\).

Let \(D\) be a square-free integer such that
\[
\gcd(D,2\cdot3\cdot5\cdot13)=1,\ D>0,\  D\equiv 1 \bmod 8.
\]
Then \(\tau_D=1\).
\end{lemma}
\begin{proof}

By Theorem~1 and Remark~1 of \cite{Kloosterman}, we have
\[
\tau_D=
\frac{  \#E_1^D(\mathbb{Q})[\phi] \,\Omega_{E_2^D}\,\prod_{p} c_{E_2^D/\mathbb{Q}_p}}
{ \#E_2^D(\mathbb{Q})[\hat{\phi}] \,\Omega_{E_1^D}\,\prod_{p} c_{E_1^D/\mathbb{Q}_p}}.
\]

By the argument in case (b) of the proof of Theorem \ref{thm:twists} the assumptions $\mathrm{gcd}(D,2\Delta_{E_1}\Delta_{E_2})=1, D\equiv 1 \bmod 8$, together with \(\Delta_{E_1}=\Delta_{E_2}\), imply $\Delta_{E_1^D}=\Delta_{E_2^D}$. By Theorem~\ref{kodaira}, we have $c_{E_1^D/\mathbb{Q}_p}=c_{E_2^D/\mathbb{Q}_p}$ for every odd prime $p$ and $E^D_i \ (i=1,2)$ have good reduction at $2$.

By Theorem~\ref{period}, $\Omega_{E_1}=\Omega_{E_2}$ since $\mathrm{rank}_{an}(E_i/\Bbb{Q})\equiv 0 \bmod 2 \ (i=1,2)$ by LMFDB. We have $\Omega_{E^D_1}=\Omega_{E^D_2}$ since $\Omega_{E^D_i}=\dfrac{\Omega_{E_i}}{\sqrt{D}}\ (i=1,2)$ follows from 
Corollary 2.6 and Theorem 3.2 in \cite{pal}. 

Moreover, $\#E_1^D(\mathbb{Q})[\phi]=\#E_2^D(\mathbb{Q})[\hat{\phi}]=2$. Therefore, we have $\tau_D=1$.

\end{proof}

\begin{remark}
One may also check this directly by computing the local
factors by Proposition \ref{tamagawaratio}:

By Proposition \ref{tamagawaratio}, $\tau_D=\prod_p \frac{1}{2} \#W_p(\phi, D)$.

For any prime $p\neq 2$, Lemma~3.8 of \cite{Sc}
(or, alternatively, Lemma~2.1(ii) of \cite{zy}), we have $\frac{1}{2} \#W_p(\phi, D)=\frac{c_{E^D_2/\Bbb{Q}_p}}{c_{E^D_1/\Bbb{Q}_p}}=1$ by Theorem \ref{kodaira}. For $p=\infty$ and $p=2$, we have $W_{\infty}(\phi,D)=W_{\infty}(\phi,1)$ and
$W_2(\phi,D)=W_2(\phi,1)$, since $D$ is a square in $\mathbb{R}^\times$ (as $D>0$)
and in $\mathbb{Q}_2^\times$ (as $D\equiv 1 \bmod 8$), respectively.
Hence it suffices to compute the local factors at $p=\infty$ and $p=2$ for $D=1$.\par 

\quad At the infinite place $p=\infty$, by Proposition~7.6 of \cite{dokchitser}, we have
$\frac{\#\mathrm{Ker}\ (E_1(\mathbb{R})\stackrel{\phi}{\to} E_2(\mathbb{R}))}
{\#\mathrm{Coker}\ (E_1(\mathbb{R})\stackrel{\phi}{\to} E_2(\mathbb{R}))}
=2$ as $a\coloneqq 25350>0, b\coloneqq 2471625>0, 4b<a^2$.
Since $\#\ker (E_1(\mathbb{R})\stackrel{\phi}{\to} E_2(\mathbb{R}))=2$, it follows that
$\frac{1}{2}\#W_{\infty}(\phi,1)
=\frac{1}{2}\#\operatorname{coker}(E_1(\mathbb{R})\stackrel{\phi}{\to} E_2(\mathbb{R}))
=\frac{1}{2}$. \quad For $p=2$, let $E_i^{\min}/\mathbb{Z}_2 (i=1,2)$ be the minimal model of $E_i$ and let
$\iota_i:E_i\to E_i^{\min}$ be the induced $\mathbb{Q}_2$-isomorphism; set
$\phi^{\min}:=\iota_2\circ\phi\circ\iota_1^{-1}$.
Then $\iota_2$ induces an isomorphism
$E_2(\mathbb{Q}_2)/\phi(E_1(\mathbb{Q}_2))\cong E_2^{\min}(\mathbb{Q}_2)/\phi^{\min}(E_1^{\min}(\mathbb{Q}_2))$,
hence $\#W_2(\phi,1)=\#W_2(\phi^{\min},1)$.
Therefore we may replace $(E_1,E_2,\phi)$ by their minimal models over $\mathbb{Z}_2$ (and keep the same notation). By Lemma~3.8 of \cite{Sc}, we have
$\#W_2(\phi,1)=\#(E_2(\mathbb{Q}_2)/\phi(E_1(\mathbb{Q}_2)))
=\#E_1(\mathbb{Q}_2)[\phi]\cdot
\left|\frac{\phi^*\omega_{E_2}}{\omega_{E_1}}\right|_{\mathbb{Q}_2}^{-1}\cdot
\frac{c_{E_2/\mathbb{Q}_2}}{c_{E_1/\mathbb{Q}_2}}$,
where $\phi^*$ denotes the pullback of a minimal differential by~$\phi$.
Moreover, $\#E_1(\mathbb{Q}_2)[\phi]\cdot
\left|\frac{\phi^*\omega_{E_2}}{\omega_{E_1}}\right|_{\mathbb{Q}_2}^{-1}\cdot
\frac{c_{E_2/\mathbb{Q}_2}}{c_{E_1/\mathbb{Q}_2}}$
is equal to $2\left|\frac{\phi^*\omega_{E_2}}{\omega_{E_1}}\right|_{\mathbb{Q}_2}^{-1}$.
Since $E_1$ has good ordinary reduction at $2$, by Proposition~4.8 of \cite{dokchitser}, we have
$\left| \frac{\phi^*\omega_{E_2}}{\omega_{E_1}}\right|_{\mathbb{Q}_2}^{-1}=2$ if $\mathrm{Ker}\phi$ lies in the formal group $\widehat{E}_1(2 \mathbb{Z}_2)$.
Let $P=(0,0)$ be the generator of $\mathrm{Ker}\phi$.
Under the change of variables from the model $y^2=x^3+25350x^2+2471625x$ to its minimal model $E_1^{\min}$,
$P$ is sent to $(x,y)=\left(\frac{8451}{4},-\frac{8451}{8}\right)$ by
$(x,y)\mapsto \left(\frac{x+8451}{4}, \frac{y-x-8451}{8}\right)$.
The local parameter $z = -x/y$ satisfies $v_2(z) = v_2(2) = 1 > 0$, which confirms that $P$ lies in $\widehat{E}_1(2\mathbb{Z}_2)$. \par

\quad Thus,
\[
\tau_D=
\left(\frac{1}{2}\#W_{\infty}(\phi,1)\right)
\left(\frac{1}{2}\#W_{2}(\phi,1)\right)
=\frac{1}{2}\times 2=1.
\]

\end{remark}

The following proposition gives an explicit way to choose a twisting parameter
\(D\) for the pair in Example~2.  The argument is not special to this numerical
pair: the essential idea is to choose \(D\) so that the local conditions
\[
W_v(\phi,D)=W_v(\phi,1)
\]
are preserved at all relevant places \(v\). In the previous version of this paper (arXiv v2), the argument was a case-by-case brute-force computation. The idea of streamlining the argument by comparing the local conditions \(W_v(\phi,D)\) with \(W_v(\phi,1)\) was suggested by the referee.

\begin{proposition}\label{pair}
Let $E_1, E_2$ be as in Lemma \ref{taud=1}, so that, in particular,
\[
\Sel^\phi(E_1)=\Sel^{\hat\phi}(E_2)=\Bbb{Z}/2\Bbb{Z}.
\]
Let $D$ be any square-free integer such that all of the following hold:
\begin{itemize}
    \item $\gcd(D,2\cdot3\cdot5\cdot13)=1$ and     
    $D>0,D\equiv 1 \bmod{8}$.    
    \item $D=1$ or $\left(\frac{65}{p}\right)=-1$ for all $p \mid D$.
\end{itemize}
Then
\[
\Sha(E_1^D/\Bbb{Q})[2]=\Sha(E_2^D/\Bbb{Q})[2]=0.
\]
\end{proposition}

\begin{proof}
Let us define \[S \coloneqq \{v \mid v \text{ divides } \deg\phi\} \cup \{v \mid v \text{ divides } \Delta_{E^D_1}\Delta_{E^D_2}\} \cup \{\infty\} = \{2, 3, 5, 13, \infty\} \cup \{p: \text{prime}\mid p\mid D\}\] and $\mathbb{Q}(S,2) \coloneqq \{\bar{d} \in \mathbb{Q}^{\times}/(\mathbb{Q}^{\times})^2 \mid v(d) \equiv 0 \bmod{2} \text{ for all } v \notin S\}$.

Via the following isomorphism, the Selmer group can be calculated by checking  the local 
solubility of torsors $C_d$. We have

\[
\mathrm{Sel}^{\phi}(E_1^D)\cong 
\left\{
  \bar{d}\in \Bbb{Q}(S,2)
  \;\middle|\;
  \begin{aligned}
    & \bar{C}_d(\Bbb{Q}_v)\neq \emptyset , \forall v\in S, \text{where } \ \bar{C}_d \ \text{is the projective closure of}\\
&\text{affine curve}\ C_d: dy^2=d^2-50700dDx^2+632736000D^2x^4 \ \text{in}\ \Bbb{P}^3   
  \end{aligned} \right\}\](see \cite{sil}, X, Remark 3.7 and Proposition 4.9). 
  
The projective closure $\bar{C}_d$ has two points at infinity $\infty=[0,0,\pm D\sqrt{\frac{632736000}{d}},1]$ in $\mathbb{P}^3$ (see \cite{sil}, X, Remark 3.7). Since $\sqrt{\frac{632736000}{65}} = 3120$, we note that \begin{equation}\label{eq:Selmer65}   
 65 \in \mathrm{Sel}^{\phi}(E_1^D/\mathbb{Q}).
\end{equation}

Since $D$ is a square in $\mathbb{Q}_p$ for all primes $p \in \{2,\infty\}$, $E_i^D$ and $E_i$ are isomorphic over $\mathbb{Q}_p$. We therefore have
\[
W_p(\phi,D)=W_p(\phi,1)
\]
for such $p$.

For any prime $p\neq 2$, Lemma~3.8 of \cite{Sc}
(or, alternatively, Lemma~2.1(ii) of \cite{zy}) gives
\[
\#W_p(\phi,D)
=
2\frac{c_{E^D_2/\mathbb{Q}_p}}{c_{E^D_1/\mathbb{Q}_p}}.
\]
By Theorem~\ref{kodaira}, we have
\[
\frac{c_{E^D_2/\mathbb{Q}_p}}{c_{E^D_1/\mathbb{Q}_p}}
=
\frac{c_{E_2/\mathbb{Q}_p}}{c_{E_1/\mathbb{Q}_p}}
=
1.
\]
We therefore obtain
\[
\#W_p(\phi,D)=2=\#W_p(\phi,1).
\]

For every odd prime \(p\in\{3,5,13\}\cup\{q:q\mid D\}\), the class
\(\operatorname{res}_p(65)\) is non-trivial in
\(\mathbb Q_p^\times/\mathbb Q_p^{\times 2}\). Indeed, for \(p=5,13\) this
follows from \(v_p(65)=1\), since squares have even valuation. For \(p=3\),
we have \(65\equiv 2\bmod 3\), which is not a square modulo \(3\). Finally,
if \(p\mid D\), then \(p\nmid 65\) by the assumption
\(\gcd(D,2\cdot3\cdot5\cdot13)=1\), and the condition
\[
\left(\frac{65}{p}\right)=-1
\]
implies that \(65\) is a nonsquare unit in \(\mathbb Q_p^\times\). 

In the remaining case for $p$, the curve $E_1^D$ has good reduction at $p$. In this case, we have
\[
W_p(\phi,1)
=
H^1_{\mathrm{nr}}(\mathbb{Q}_p,E[\phi])
=
H^1_{\mathrm{nr}}(\mathbb{Q}_p,E^D[\phi])
=
W_p(\phi,D).
\]

Since $65 \in \mathrm{Sel}^{\phi}(E_1^D)$, we have $\mathrm{res}_{p}(65) \in W_p(\phi,D)$. Here, we regard \(W_p(\phi,D)\) and
\(W_p(\phi,1)\) as subgroups of the same cohomology group, via the
natural identification
\[
E_1^D[\phi]\simeq E_1[\phi].
\]
Equivalently, we identify
\[
H^1(\mathbb{Q}_p,E_1^D[\phi])
\simeq
H^1(\mathbb{Q}_p,E_1[\phi])
\simeq
\mathbb{Q}_p^\times/\mathbb{Q}_p^{\times 2}.
\]

Therefore, we have $W_p(\phi,D)=\{1, \mathrm{res}_p(65)\}=W_p(\phi,1)$ in $H^1(\Bbb{Q}_p,E[\phi])\cong {\Bbb{Q}_p}^{\times}/{{\Bbb{Q}_p}^{\times}}^{2}$.

It follows that for all primes $p$, we have $W_p(\phi,D)=W_p(\phi,1)$. Thus
\[
\Sel^\phi(E_1^D)=\Sel^\phi(E_1)=\Bbb{Z}/2\Bbb{Z}.
\]
Computation with the Tamagawa ratio (Lemma \ref{taud=1}) shows the same is true for $\Sel^{\hat\phi}(E_2^D)$.

By the same argument as (\ref{shazero}) in the proof of Proposition \ref{twistsha}, we have $\Sha(E^D_1/\Bbb{Q})[2]=\Sha(E^D_2/\Bbb{Q})[2]=0$ and $\mathrm{rank}(E_1^D/\Bbb{Q})=\mathrm{rank}(E_2^D/\Bbb{Q})=0$.

\end{proof}

\subsection{Pairs of elliptic curves sharing the BSD invariants, Kodaira symbols and minimal discriminants}
\vskip\baselineskip

Consider the pair \(E_1: y^2=x^3+25350x^2+2471625x\) and \(E_2: y^2=x^3-50700x^2+632736000x\) described in Example~\ref{example}.

Let \(D>0\) be a square-free integer such that \(D\equiv 1 \bmod{8}\),
\(\gcd(D,2\cdot 3\cdot 5\cdot 13)=1\), and either \(D=1\) or
\(\left(\frac{65}{p}\right)=-1\) for every prime \(p\mid D\).
Assume further that \(D\) is sufficiently large so that
\(E_1^D(\mathbb{Q})_{\mathrm{tors}}\) and
\(E_2^D(\mathbb{Q})_{\mathrm{tors}}\) have no elements of order greater than \(2\).

By Propositions~\ref{sufficient} and~\ref{pair}, for such \(D\) the BSD invariants of \(E_1^D/\mathbb{Q}\) and \(E_2^D/\mathbb{Q}\) coincide; in addition, the Kodaira symbols agree at every prime and \(\Delta_{E_1^D}=\Delta_{E_2^D}\). We provide a table of pairs of non-isomorphic elliptic curves that share these BSD invariants and minimal discriminants.

{\small
\begin{table}[htbp]
\centering
\label{tab:elliptic_curves_38025_b}
\begin{tabular}{lll}
    \toprule
    \textbf{Elliptic curve} & $E^D_1: y^2=x^3+25350Dx^2+2471625D^2x$ & $E^D_2: y^2=x^3-50700Dx^2+632736000D^2x$ \\
    \midrule
 $j$-invariant& $257^3$  & $17^3$  \\  
Minimal discriminant & $3^6\cdot 5^9\cdot 13^9 \cdot D^6$ & $3^6\cdot 5^9\cdot 13^9 \cdot D^6$ \\

    Mordell--Weil group & 
    $\mathbb{Z}/2\mathbb{Z}$ & $\mathbb{Z}/2\mathbb{Z}$ \\
    Regulator & 1 & 1 \\
    Real period & $\dfrac{0.209\cdots}{\sqrt{D}}$ & $\dfrac{0.209\cdots}{\sqrt{D}}$ \\
    Tamagawa number & $2(p=3),2(p=5), 2(p=13), 2(p\mid D)$ & $2(p=3), 2(p=5), 2(p=13), 2(p \mid D)$\\

   Kodaira symbol  & $I_0^*(p=3), III^*(p=5), III^*(p=13)$,  & $I_0^*(p=3), III^*(p=5), III^*(p=13)$ \\

    &  $I_0^* (p \mid D)$ &  $I_0^* (p \mid D)$\\        
 \rowcolor{gray!20}  $\Sha(E^D_i/\mathbb{Q})[2^{\infty}]$ & 0 & 0 \\
 \rowcolor{gray!20} $\Sha(E^D_i/\mathbb{Q})$ & $\dagger$ & $\dagger$ \\    
    \bottomrule
\end{tabular}
\caption{Pairs of elliptic curves sharing the BSD invariants, Kodaira symbols and minimal discriminants}
\end{table}
}

\FloatBarrier

For computations of Tamagawa numbers and Kodaira symbols, see, for example, \cite{aj}. Let \(p\mid D\). Since \(D\) is square-free, we have \(v_p(D)=1\). Moreover, because \(E_1\) has good reduction at \(p\), the quadratic twist \(E_1^D\) has reduction type \(I_0^*\) at \(p\). Hence \(c_{E_1^D/\mathbb{Q}_p}=1+\#\{\alpha\in\mathbb{F}_p\mid P_D(\alpha)=0\}\), where \(P_D(T)=T(T^2+25350uT+2471625u^2)\) and \(u=D/p\). Let \(f(T)=T^2+25350uT+2471625u^2\). Then \(\left(\frac{\operatorname{disc}(f)}{p}\right)=\left(\frac{65}{p}\right)=-1\). Therefore \(f\) has no root in \(\mathbb{F}_p\), while \(T=0\) gives the unique root of \(P_D(T)\). Thus \(c_{E_1^D/\mathbb{Q}_p}=1+1=2\).

For $\dagger$, note that $\Sha(E^D_1/\mathbb{Q})\cong \Sha(E^D_2/\mathbb{Q})$ as abelian groups by Proposition \ref{sha} (2), and they need not be zero as the following example shows.
 \par

\begin{example}\label{const}
 
 BSD conjecture predicts that there exists a pair $(E^D_1, E^D_2)$ that has the same BSD data and nontrivial Tate-Shafarevich groups.

Indeed, 

\[
  \Sha(E^{17}_1/\mathbb{Q}) \cong \Sha(E^{17}_2/\Bbb{Q})\cong  (\mathbb{Z}/3\mathbb{Z})^2,
  \qquad\text{since}\qquad
  \frac{L(E_1^{17},1)}{4\,\Omega_{E_1^{17}}} = 9\] by SageMath.

To find many $D$ such that
\[
\Sha(E_1^D/\mathbb{Q})[3]\cong \Sha(E_2^D/\mathbb{Q})[3]\neq 0,
\]
see \cite{shiga2026} for an attempt in terms of the theory of visibility.

\end{example}

\section{Acknowledgements}
I would like to express my deepest gratitude to Prof.\ Nobuo Tsuzuki for his many constructive discussions and continual encouragement throughout this work. I would like to express my sincere thanks to Prof.\ Christian Wuthrich for insightful discussions at Y-RANT in Nottingham. I am grateful to the anonymous referees for their very careful reading of the manuscript and for their valuable comments.

\FloatBarrier

\end{document}